\documentclass[11pt,reqno]{amsart}
\usepackage{amssymb,latexsym}
\usepackage{graphicx}
\usepackage{url}

\theoremstyle{plain}
\numberwithin{equation}{section}
\newtheorem{thm}{Theorem}[section]
\newtheorem{theorem}[thm]{Theorem}

\newtheorem{lemma}[thm]{Lemma}

\newtheorem{proposition}[thm]{Proposition}

\begin{document}

\setcounter{page}{1}

\title[Proofs of some binomial identities using the method of last squares]{Proofs of some binomial identities using the method of last squares}
\author{Mark Shattuck}
\address{Department of Mathematics\\
         University of Tennessee\\
         Knoxville \\
         TN 37996-1300, USA}
\email{shattuck@math.utk.edu}
\author{Tam\'{a}s Waldhauser}
\address{Mathematics Research Unit\\
         University of Luxembourg\\
         6 rue Richard Coudenhove-Kalergi\\
         L-1359 Luxembourg, Luxembourg, and\\
         Bolyai Institute\\
         University of Szeged\\
         Aradi v\'{e}rtan\'{u}k tere 1\\
         H-6720 Szeged, Hungary}
\email{twaldha@math.u-szeged.hu}

\begin{abstract}
We give combinatorial proofs for some identities involving binomial sums that have no closed form.
\end{abstract}

\dedicatory{Dedicated to P\'{e}ter Hajnal on his fiftieth birthday}

\maketitle

\section{Introduction}

The main result of this paper is a combinatorial proof of the following
identity for $0\leq r\leq\frac{m}{2}-1$:%
\[
\sum_{i=r+1}^{\left\lfloor \frac{m}{2}\right\rfloor }\binom{m}{2i}\binom
{i-1}{r}=2^{m-1-2r}\sum_{k=0}^{\left\lfloor \frac{r}{2}\right\rfloor }%
\binom{m-3-r-2k}{r-2k}+\left(  -1\right)  ^{r+1}\text{.}%
\]
The sum on the left-hand side appeared in connection with investigations
about the arity gap of polynomial functions \cite{GapInfC}. There, only the fact that this sum
is always odd was needed, which is not hard to prove by induction. Clearly,
the right-hand side reveals a much stronger divisibility property.

In the course of the proof we will give three other expressions for the same
sum. Before presenting these, let us introduce the following notation.%
\begin{align*}
S_{n,r} &  =\sum_{i=r+1}^{\left\lfloor \frac{n}{2}\right\rfloor }\binom{n}%
{2i}\binom{i-1}{r}\\
T_{n,r} &  =\sum_{j=r+1}^{n}\binom{n}{j}\binom{j-1}{r}\\
U_{n,r} &  =\sum_{j=r+1}^{n}\binom{j-1}{r}2^{j-1-r}\\
V_{n,r} &  =\sum_{j=1}^{n-r}\binom{n-1-j}{r-1}2^{n-r-j}\left(2^{j}-1\right)\\
W_{n,r} &  =2^{n-r}\sum_{k=0}^{\left\lfloor \frac{r}{2}\right\rfloor }%
\binom{n-2-2k}{r-2k}+\left(  -1\right)  ^{r+1}%
\end{align*}
We will prove the following identities relating these five sums.

\begin{theorem}
\label{thm main}For all $0\leq r\leq\frac{m}{2}-1$, we have
\[
S_{m,r}=T_{m-1-r,r}=U_{m-1-r,r}=V_{m-1-r,r}=W_{m-1-r,r}.
\]
\end{theorem}

Nowadays, such identities can be proven automatically thanks to the machinery
developed by Petkov\v{s}ek, Wilf and Zeilberger \cite{A=B}, but we believe that the problem of finding combinatorial proofs in the spirit of \cite{Ben} is still of interest. Let us mention that the above
sums have no closed form. Indeed, considering, e.g., $f\left(  n\right)
=T_{2n,n}$, creative telescoping \cite{Zb,Ekhad} finds the recurrence%
\[
\left(  24n^{2}+44n+16\right)\!f\left(  n\right)  +\left(
21n^{2}+37n+14\right)\!f\left(  n+1\right)  -\left( 3n^{2}+7n+2\right)\!f\left(  n+2\right)  =0,
\]
and algorithm Hyper \cite{Hyper} shows that the only hypergeometric solutions of this
recurrence are the functions of the form $f\left(  n\right)  =c\left(
-1\right)  ^{n}$. Clearly, $T_{2n,n}$ is not such a function, hence it does
not have a hypergeometric closed form. This implies that $T_{n,r}$ and
$T_{m-1-r,r}$ do not have closed forms either, and then
Theorem~\ref{thm main} shows that the other four sums also do not have closed forms. However, $W_{m-1-r,r}$ stands out from the five expressions, since it
is the only one where the number of summands is independent of $m$; hence it
may be regarded as a closed form, if $m$ is considered as the only variable (with
$r$ regarded as a parameter).  Furthermore, one can show that the five expressions in Theorem~\ref{thm main} have common generating function
\[
\sum_{m \geq 2r+2}S_{m,r}x^m=\frac{x^{2r+2}}{(1-x)(1-2x)^{r+1}},
\]
for fixed $r \geq 0$.

Let us also note that replacing $i-1$ by $i$ in $S_{m,r}$ yields a simple closed form and the resulting
identity is one of the well-known Moriarty formulas (see, e.g., \cite{Gould,Sh}):%
\[
\sum_{i=r}^{\left\lfloor \frac{m}{2}\right\rfloor }\binom{m}{2i}\binom{i}%
{r}=2^{m-1-2r}\binom{m-r}{r}\frac{m}{m-r}.
\]
Similarly, replacing $j-1$ by $j$ in $T_{n,r}$, we get the easy-to-prove
identity%
\[
\sum_{j=r}^{n}\binom{n}{j}\binom{j}{r}=2^{n-r}\binom{n}{r}.
\]

The following table shows the value of $T_{n,r}$ for $n=1,\ldots,10$ and $r=0,\ldots,9$.

\medskip

\begin{center}
\renewcommand{\arraystretch}{1.2}%
\begin{tabular}
[c]{|r||r|r|r|r|r|r|r|r|r|r|}\hline
& $0$ & $1$ & $2$ & $3$ & $4$ & $5$ & $6$ & $7$ & $8$ & $9$\\\hline\hline
$1$ & $1$ &  &  &  &  &  &  &  &  & \\\hline
$2$ & $3$ & $1$ &  &  &  &  &  &  &  & \\\hline
$3$ & $7$ & $5$ & $1$ &  &  &  &  &  &  & \\\hline
$4$ & $15$ & $17$ & $7$ & $1$ &  &  &  &  &  & \\\hline
$5$ & $31$ & $49$ & $31$ & $9$ & $1$ &  &  &  &  & \\\hline
$6$ & $63$ & $129$ & $111$ & $49$ & $11$ & $1$ &  &  &  & \\\hline
$7$ & $127$ & $321$ & $351$ & $209$ & $71$ & $13$ & $1$ &  &  & \\\hline
$8$ & $255$ & $769$ & $1023$ & $769$ & $351$ & $97$ & $15$ & $1$ &  & \\\hline
$9$ & $511$ & $1793$ & $2815$ & $2561$ & $1471$ & $545$ & $127$ & $17$ & $1$ &
\\\hline
$10$ & $1023$ & $4097$ & $7423$ & $7937$ & $5503$ & $2561$ & $799$ & $161$ &
$19$ & $1$\\\hline
\end{tabular}
\end{center}

\medskip

\noindent This table appears in OEIS (up to signs and other minor alterations) as A118801, A119258 and A145661 \cite{OEIS}. The formula given for A118801
is equivalent to $W_{n,r}$, while the formula given for A119258 is equivalent to $T_{n,r}$.

The proof of Theorem~\ref{thm main} will be presented in the next section as a
sequence of six propositions. First we define certain arrangements of dominos
and squares, tiling a $1\times m$ board, and prove that the number of such arrangements is $S_{m,r}$ (see Proposition~\ref{prop S} below).
Then we define another kind of arrangement, where we tile a $1\times n$ board by
three kinds of squares, and show that $T_{n,r},\ U_{n,r}$ and $V_{n,r}$ count the number of such arrangements (Propositions~\ref{prop T}, \ref{prop U} and \ref{prop V}). In Proposition~\ref{prop S=T} we give a
bijection between the two kinds of arrangements with $n=m-1-r$, thereby
proving the identity $S_{m,r}=T_{m-1-r,r}$. Finally, we consider $W_{n,r}$: in
Lemma~\ref{lemma conjugation}, perhaps the trickiest
part of the proof, we give a bijection between two special subsets of
arrangements, and in Proposition~\ref{prop T=W} we use this bijection to prove the
identity $T_{n,r}=W_{n,r}$. The title of the paper is explained by the fact that several times in the course of the proof, squares towards the right end of the board
(e.g., squares after the last domino) will play a crucial role.
\subsubsection*{Acknowledgments.}
The second named author acknowledges that the present project is supported by the National Research Fund, Luxembourg, and cofunded
under the Marie Curie Actions of the European Commission (FP7-COFUND), and supported by
the Hungarian National Foundation for Scientific Research under grant no. K77409.

\section{Proofs}

We will consider coverings of a board of length $m$ (i.e., a
$1\times m$ \textquotedblleft chessboard\textquotedblright) by dominos and
white and black squares:

\begin{center}
\includegraphics[height=0.322in]{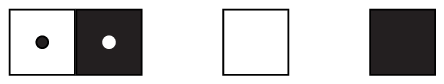}
\end{center}

\noindent Each domino covers two consecutive cells of the board, and the dominos may not be turned around: the white part of the domino
is always on the left. We will refer to such coverings as arrangements,
and we will denote by $\mathcal{D}_{m,r}$ the set of all arrangements
containing $r$ dominos and $m-2r$ squares such that the first (leftmost)
cell of the board is covered by a black square. We partition this set into
two subsets depending on the colors of the \emph{last squares}, i.e., the
colors of the squares to the right of the last (rightmost) domino: let $\mathcal{D}%
_{m,r}^{-}\subseteq\mathcal{D}_{m,r}$ denote the set of those arrangements,
where all squares to the right of the last domino are black (if any), and let
$\mathcal{D}_{m,r}^{+}=\mathcal{D}_{m,r}\setminus\mathcal{D}_{m,r}^{-}$ denote
the set of those arrangements where there is at least one white square to the
right of the last domino (not necessarily immediately adjacent to the domino).
Here is an example of an arrangement belonging to $\mathcal{D}_{17,3}^{+}$:

\begin{center}
\includegraphics[width=5.5in]{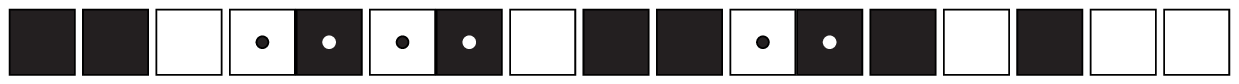}
\end{center}

In the following proposition we count the arrangements of $\mathcal{D}%
_{m,r}^{+}$.

\begin{proposition}\label{prop S}
For all $0\leq r\leq\frac{m}{2}-1$, we have $\left\vert \mathcal{D}_{m,r}%
^{+}\right\vert =S_{m,r}$.
\end{proposition}

\begin{proof}
We give an interpretation for the sum $S_{m,r}$ that is in a one-to-one
correspondence with the arrangements in $\mathcal{D}_{m,r}^{+}$. First we
choose $2i$ squares of our board of length $m$ (in this example $m=17$ and
$i=5$):

\begin{center}
\includegraphics[width=5.5in]{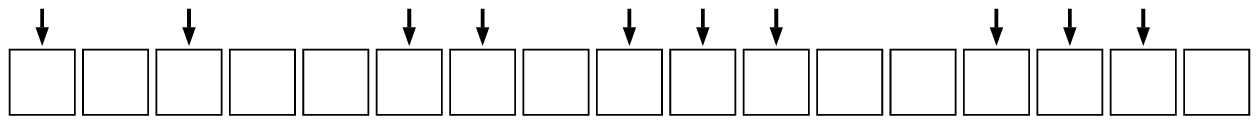}
\end{center}

\noindent Then we color the squares of the board one by one from left to right,
starting with black on the first square, and changing the color after every
chosen square:

\begin{center}
\includegraphics[width=5.5in]{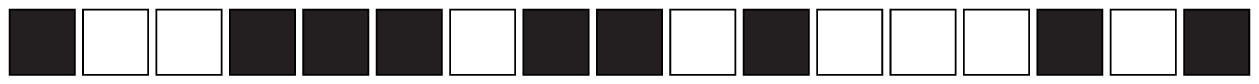}
\end{center}

\noindent There are $i-1$ or $i$ places on the board where the color is
changing from white to black (going from left to right), depending on whether or not the
last square was among the $2i$ chosen squares (the above example corresponds
to the second case). We choose $r$ of these places with the restriction that in
the second case we are not allowed to choose the last place (in this example
$r=2$):

\begin{center}
\includegraphics[width=5.5in]{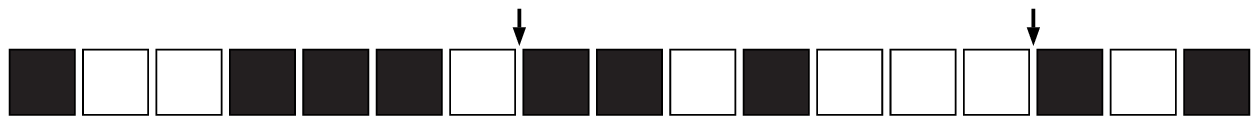}
\end{center}

Clearly, the number of such colored boards with $r$ marks over some
white-to-black color changes is $S_{m,r}$. Putting dominos in the $r$ marked
places such that the middle of each domino is exactly at the place where the
color changes from white to black (and removing the marks), we obtain an
arrangement belonging to $\mathcal{D}_{m,r}^{+}$:

\begin{center}
\includegraphics[width=5.5in]{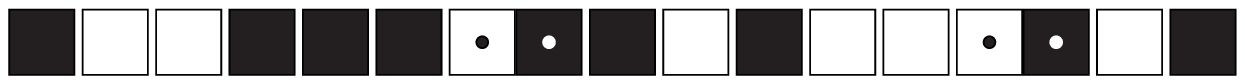}
\end{center}

\noindent This assignment is a bijection: to obtain the inverse, just put a
mark over the middle of each domino, and then remove the dots from the dominos.
\end{proof}

We will interpret $T_{n,r},\ U_{n,r}$ and $V_{n,r}$ using another arrangement. Let
us cover a board of length $n$ with three kinds of squares:\ white
squares, black squares and (white) squares decorated by a triangle (for
brevity, we will refer to the latter as a decorated square):

\begin{center}
\includegraphics[height=0.322in]{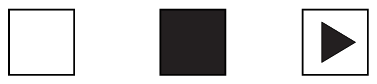}
\end{center}

\noindent Let $\mathcal{B}_{n,r}$ denote the set of those arrangements where
the number of black squares is $r$, and the final (rightmost) square is not
black (i.e., it is either white or decorated). Just as in the case of $\mathcal{D}_{m,r}$, we
consider the \emph{last squares}, namely those squares to the right of the
last black square: let $\mathcal{B}_{n,r}^{-}\subseteq\mathcal{B}_{n,r}$
denote the set of those arrangements, where all squares to the right of the
last black square are white (if any), and let $\mathcal{B}_{n,r}%
^{+}=\mathcal{B}_{n,r}\setminus\mathcal{B}_{n,r}^{-}$ denote the set of those
arrangements where there is at least one decorated square to the right of the
last black square (not necessarily immediately adjacent to the black square). Here
is an example of an arrangement belonging to $\mathcal{B}_{17,6}^{+}$:

\begin{center}
\includegraphics[width=5.5in]{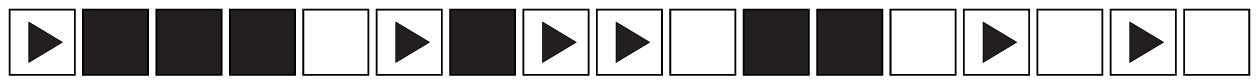}
\end{center}

Let us define the \emph{weight} of an arrangement in $\mathcal{B}%
_{n,r}$ as follows. If the second-to-last square of the board is not black, then
the weight is $0$ (recall that the last square is never black). Otherwise, the
weight is the length of the interval of consecutive black squares ending at
the second-to-last square of the board. The example above is of weight $0$, and the arrangement $\alpha$
appearing in the proof of Lemma~\ref{lemma conjugation} below is of weight $3$. Let
$\mathcal{B}_{n,r}^{\text{e}},\mathcal{B}_{n,r}^{\text{o}}$ denote the set of
arrangements of even, odd weight, respectively, and let us define the sets
$\mathcal{B}_{n,r}^{+,\text{e}},\mathcal{B}_{n,r}^{+,\text{o}},\mathcal{B}%
_{n,r}^{-,\text{e}},\mathcal{B}_{n,r}^{-,\text{o}}$ by taking $\mathcal{B}_{n,r}^{+,\text{e}}=\mathcal{B}_{n,r}^{+}\cap\mathcal{B}_{n,r}^{\text{e}}$, etc.

In the following three propositions, we count the arrangements in $\mathcal{B}_{n,r}^{+}$ in three different ways,
thereby proving the identity $T_{n,r}=U_{n,r}=V_{n,r}$.

\begin{proposition}\label{prop T}
For all $0\leq r\leq n-1$, we have $\left\vert \mathcal{B}_{n,r}^{+}\right\vert
=T_{n,r}$.
\end{proposition}

\begin{proof}
Choose $j$ cells from the board, cover all other cells by white squares,
cover the last one of the chosen cells by a decorated square, and then put $r$
black squares and $j-1-r$ decorated squares on the remaining $j-1$ chosen cells.
\end{proof}

\begin{proposition}\label{prop U}
For all $0\leq r\leq n-1$, we have $\left\vert \mathcal{B}_{n,r}^{+}\right\vert
=U_{n,r}$.
\end{proposition}

\begin{proof}
We claim that the summand of $U_{n,r}$ counts those arrangements in $\mathcal{B}_{n,r}^{+}$ where the last decorated square appears on cell $j$.  First let us observe that the squares to the right of the last decorated square are all white, by the definition of $\mathcal{B}_{n,r}^{+}$.
Thus, we may choose the $r$ black squares from the $j-1$ squares to the left of the last decorated square in $\binom{j-1}{r}$ many ways, and then we may
decorate the squares in an arbitrary subset of the remaining $j-1-r$ squares in $2^{j-1-r}$ many ways.
\end{proof}

\begin{proposition}\label{prop V}
For all $0\leq r\leq n-1$, we have $\left\vert \mathcal{B}_{n,r}%
^{+}\right\vert =V_{n,r}$.
\end{proposition}

\begin{proof}
We claim that the summand of $V_{n,r}$ counts those arrangements in
$\mathcal{B}_{n,r}^{+}$ where the last black square appears on cell $n-j$; note that by the definition of $\mathcal{B}_{n,r}^{+}$, we
have $j\geq1$.  The preceding $r-1$ black squares can be chosen in $\binom{n-1-j}{r-1}$ ways.
The remaining $n-r$ squares can then either be white or decorated, with the restriction that
at least one of the $j$ squares to the right of the last black square has to be decorated.
Thus we can determine the white and decorated squares in $2^{n-r-j}\left( 2^{j}-1 \right)$ many ways, so the total number of possibilities is%
\[
\binom{n-1-j}{r-1}\cdot2^{n-r-j}\cdot\left(  2^{j}-1\right)  ,
\]
as claimed.
\end{proof}

The next proposition relates the two kinds of arrangements considered so far
and proves $S_{m,r}=T_{m-1-r,r}$.

\begin{proposition}\label{prop S=T}
For all $0\leq r\leq\frac{m}{2}-1$, we have $\left\vert \mathcal{D}_{m,r}%
^{+}\right\vert =\left\vert \mathcal{B}_{m-1-r,r}^{+}\right\vert $.
\end{proposition}

\begin{proof}
We construct a bijection from $\mathcal{D}_{m,r}^{+}$ to $\mathcal{B}%
_{m-1-r,r}^{+}$ as follows. An arrangement in $\mathcal{D}_{m,r}^{+}$
naturally divides the board into black and white intervals (regarding a domino
as a white square followed by a black square). Let us mark the first square of
each interval:

\begin{center}
\includegraphics[width=5.5in]{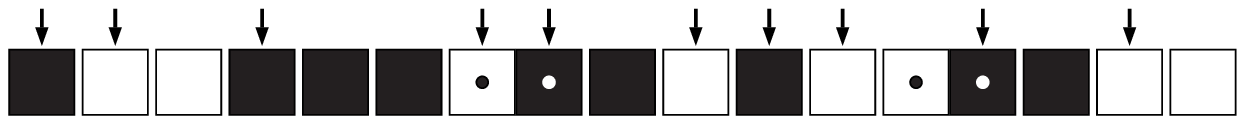}
\end{center}

\noindent Let us then replace each marked square by a decorated square unless it is
part of a domino (the right half of a domino is always marked, the left half
may be marked or unmarked), and replace each remaining black square by a white
square, unless it is part of a domino:

\begin{center}
\includegraphics[width=5.5in]{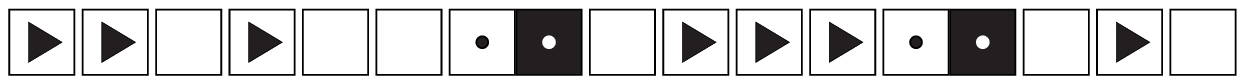}
\end{center}

\noindent Clearly, the original coloring can be recovered from this new
arrangement. Finally, we remove the first square of the board, the left half of
each domino, and the white dots from the right halves of the dominos:

\begin{center}
\includegraphics[height=0.322in]{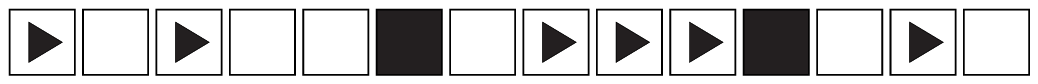}
\end{center}

\noindent This new arrangement belongs to $\mathcal{B}_{m-1-r,r}^{+}$, since the
first white square after the last domino in the original arrangement becomes a
decorated square in the new arrangement.

The above construction is indeed a bijection and its inverse can be constructed
as follows. Given an arrangement in $\mathcal{B}_{m-1-r,r}^{+}$, replace each
black square by a domino, add a new black square to the left end of the board,
and color the squares (outside the dominos) from left to right, changing the color at
each decorated square.
\end{proof}

It remains to prove that $T_{n,r}=W_{n,r}$. The key ingredient for the proof
is given by the following lemma.

\begin{lemma}\label{lemma conjugation}
For all $0\leq r\leq n-1$, we have $\left\vert \mathcal{B}_{n,r}%
^{+,\text{\textup{o}}}\right\vert =\left\vert \mathcal{B}_{n,r}%
^{-,\text{\textup{e}}}\right\vert +\left(  -1\right)  ^{r+1}$.
\end{lemma}

\begin{proof}
We give an \textquotedblleft almost bijection\textquotedblright\ between the
sets $\mathcal{B}_{n,r}^{+,\text{o}}$ and $\mathcal{B}_{n,r}^{-,\text{e}}$,
leaving one arrangement out from $\mathcal{B}_{n,r}^{+,\text{o}}$ if $r$ is
odd, and leaving one arrangement out from $\mathcal{B}_{n,r}^{-,\text{e}}$ if
$r$ is even. Let us consider an arrangement $\alpha\in\mathcal{B}_{n,r}^{+,\text{o}}\cup\mathcal{B}_{n,r}^{-,\text{e}}$
of weight $k$, and let us examine its last squares. The very last square
(i.e., the rightmost square of the board) is either white or decorated.
Before that, there is a sequence of $k$ black squares; let
us denote the first (leftmost) one of these squares by $B$. If $k=0$, then let
us define $B$ to be the last square of the board (no matter whether it is white or decorated).
Walking from square $B$ to the left, let
us denote the first non-white (i.e., either black or decorated) square by $A$,
provided there is such a square:

\begin{center}
\includegraphics[width=5.5in]{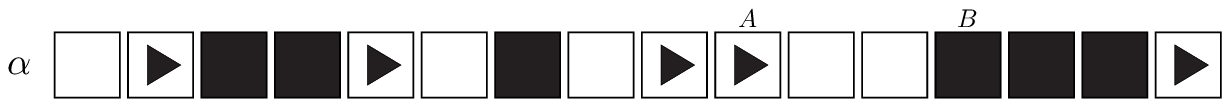}
\end{center}

The \emph{conjugate arrangement }$\overline{\alpha}$ is constructed
in the following way. If $A$ is a decorated square, then we replace $A$ by a
black square and $B$ by a white square. If $A$ is a black square, then we
replace $A$ by a decorated square and replace the white square preceding $B$
by a black square ($B$ remains unchanged). In addition, in both cases we
change the last square of the board: if it is a white square, then we change
it to a decorated square; if it is a decorated square, then we change it to a
white square. The arrangement $\alpha$ in the above example corresponds to the first case with $k=3$ ($\alpha\in\mathcal{B}_{n,r}^{+,\text{o}}$):

\begin{center}
\includegraphics[width=5.5in]{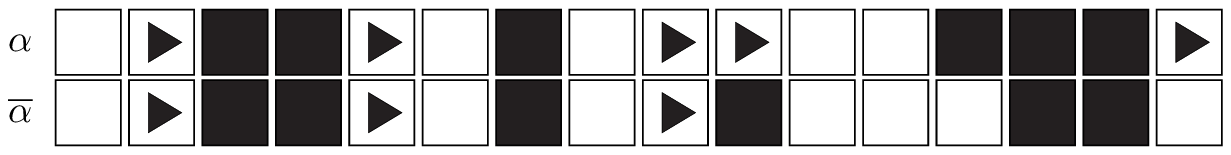}
\end{center}

\noindent Another example illustrating the second case with $k=0$ ($\beta\in\mathcal{B}_{n,r}^{-,\text{e}}$):

\begin{center}
\includegraphics[width=5.5in]{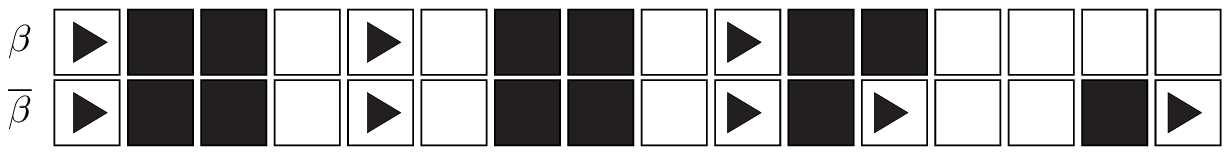}
\end{center}

The conjugate arrangement is not defined if square $A$ does not exist, i.e.,
if $k=r$ and all the squares to the left of the black squares are white.
There is only one such arrangement in $\mathcal{B}_{n,r}^{+,\text{o}}\cup\mathcal{B}_{n,r}^{-,\text{e}}$,
namely the arrangement $\varepsilon^{+}\in\mathcal{B}_{n,r}^{+,\text{o}}$ below if $r$ is odd (here, $r=5$)
and the arrangement $\varepsilon^{-}\in\mathcal{B}_{n,r}^{-,\text{e}}$ below if $r$ is even (here, $r=6$):

\begin{center}
\includegraphics[width=5.5in]{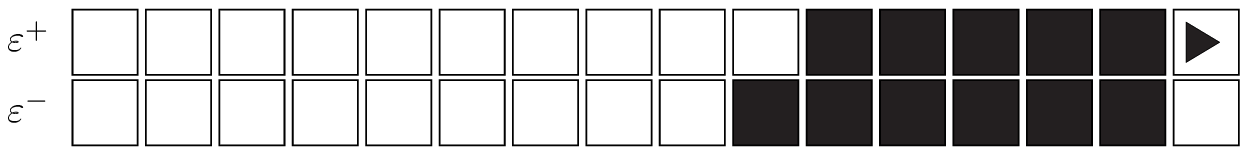}
\end{center}

\noindent Conjugation is a permutation of order two on the set $\mathcal{B}%
_{n,r}^{+,\text{\textup{o}}}\cup\mathcal{B}_{n,r}^{-,\text{\textup{e}}%
}\setminus\left\{  \varepsilon^{+},\varepsilon^{-}\right\}  $ that changes the
parity of the weight, and it also changes the \textquotedblleft
sign\textquotedblright of the arrangement\footnote{This is actually true for all arrangements in
$\mathcal{B}_{n,r}\setminus\left\{\varepsilon^{+},\varepsilon^{-}\right\}$ except for the
\textquotedblleft positive\textquotedblright ones of weight $0$.}.
Therefore, if $r$ is odd, then
conjugation provides a bijection between $\mathcal{B}_{n,r}^{+,\text{o}}%
\setminus\left\{  \varepsilon^{+}\right\}  $ and $\mathcal{B}_{n,r}%
^{-,\text{e}}$, hence $\left\vert \mathcal{B}_{n,r}^{+,\text{\textup{o}}%
}\right\vert =\left\vert \mathcal{B}_{n,r}^{-,\text{\textup{e}}}\right\vert
+1$. Similarly, if $r$ is even, then conjugation provides a bijection between
$\mathcal{B}_{n,r}^{+,\text{o}}$ and $\mathcal{B}_{n,r}^{-,\text{e}}%
\setminus\left\{  \varepsilon^{-}\right\}  $, hence $\left\vert \mathcal{B}%
_{n,r}^{+,\text{\textup{o}}}\right\vert =\left\vert \mathcal{B}_{n,r}%
^{-,\text{\textup{e}}}\right\vert -1$.
\end{proof}

\begin{proposition}\label{prop T=W}
For all $0\leq r \leq n-1$, we have $T_{n,r}=W_{n,r}$.
\end{proposition}

\begin{proof}
We may express $T_{n,r}$ with the aid of the previous lemma:%
\[
T_{n,r}=\left\vert \mathcal{B}_{n,r}^{+}\right\vert =\left\vert \mathcal{B}%
_{n,r}^{+,\text{\textup{e}}}\right\vert +\left\vert \mathcal{B}_{n,r}%
^{+,\text{\textup{o}}}\right\vert =\left\vert \mathcal{B}_{n,r}%
^{+,\text{\textup{e}}}\right\vert +\left\vert \mathcal{B}_{n,r}%
^{-,\text{\textup{e}}}\right\vert +\left(  -1\right)  ^{r+1}=\left\vert
\mathcal{B}_{n,r}^{\text{\textup{e}}}\right\vert +\left(  -1\right)  ^{r+1}.
\]
It remains to prove that%
\[
\left\vert \mathcal{B}_{n,r}^{\text{\textup{e}}}\right\vert =\sum
_{k=0}^{\left\lfloor \frac{r}{2}\right\rfloor }2^{n-r}\binom{n-2-2k}{r-2k}.
\]

We claim that the summand counts the arrangements in $\mathcal{B}_{n,r}$ of
weight $2k$. Such an arrangement can be built as follows. First we put an
interval of $2k$ black squares on the board such that the last one of these
black squares is the second-to-last square of the board. Then we have $n-2-2k$
places where we can put the remaining $r-2k$ black squares:

\begin{center}
\includegraphics[width=5.5in]{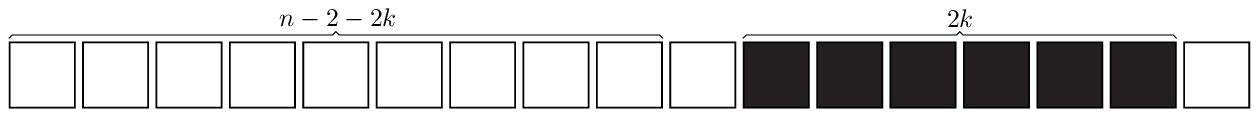}
\end{center}

\noindent Thus there are $\binom{n-2-2k}{r-2k}$ possibilities regarding the
placement of the black squares, and each one of the remaining $n-r$ squares
can be either white or decorated, hence the number of arrangements of weight
$2k$ is indeed%
\[
2^{n-r}\binom{n-2-2k}{r-2k}.
\]

\end{proof}

\medskip

\noindent AMS Classification Numbers: 05A19, 11B65
\end{document}